\documentclass{amsart}

\fussy

\usepackage{mathptmx}      
\usepackage{amsmath}
\usepackage{amsthm}

\newcommand{\BLOCK}[1]{\left[{#1}\right]}                           

\newcommand{\relb}[3]{{#1}\;{#2},\:{#3}}                            
\newcommand{\relu}[2]{{#1}\;{#2}}                                   

\DeclareMathOperator{\Q}{Q}                                         
\DeclareMathOperator{\R}{R}                                         
\newcommand{\RQ}{\R_{\Q}}                                           

\newcommand{\dit}[1]{\overline{#1}}                                 
\newcommand{\xbar}{\dit{x}}                                         
\newcommand{\ybar}{\dit{y}}                                         
\newcommand{\zbar}{\dit{z}}                                         

\newcommand{\apl}[1]{(#1)}                                          
\newcommand{\app}[2]{\apl{#1\,#2}}                                  
\newcommand{\appp}[3]{\apl{#1\,#2\,#3}}                             
\newcommand{\apppp}[4]{\apl{#1\,#2\,#3\,#4}}                        
\newcommand{\appppp}[5]{\apl{#1\,#2\,#3\,#4\,#5}}                   
\newcommand{\appppppp}[7]{\apl{#1\,#2\,#3\,#4\,#5\,#6\,#7}}         

\DeclareMathOperator{\FORALL}{\forall}                              
\DeclareMathOperator{\EXISTS}{\exists}                              
\DeclareMathOperator{\EXISTSONE}{\dot{\exists}}                     

\DeclareMathOperator{\AND}{\wedge}                                  
\DeclareMathOperator{\OR}{\vee}                                     
\DeclareMathOperator{\NOT}{\neg}                                    
\DeclareMathOperator{\IMPLIES}{\Longrightarrow}                     
\DeclareMathOperator{\EQUALS}{\Longleftrightarrow}                  

\newcommand{\set}[1]{#1}                                            
\DeclareMathOperator{\IN}{\sqsubseteq}                              
\DeclareMathOperator{\SUBSET}{\subseteq}                            

\newcommand{\ax}[1]{_{\text{\tiny{(Ax~\ref{#1})}}}}                 
\newcommand{\axs}[2]{_{\text{\tiny{(Axs~\ref{#1},~\ref{#2})}}}}     
\newcommand{\df}[1]{_{\text{\tiny{(Df~\ref{#1})}}}}                 
\newcommand{\lm}[1]{_{\text{\tiny{(Lm~\ref{#1})}}}}                 
\newcommand{\pr}[1]{_{\text{\tiny{(Pr~\ref{#1})}}}}                 

\newtheorem{axiom}{Axiom}[section]
\newtheorem{corollary}{Corollary}[section]
\newtheorem{definition}{Definition}[section]
\newtheorem{example}{Example}[section]
\newtheorem{lemma}{Lemma}[section]
\newtheorem{proposition}{Proposition}[section]

\newcommand{\proofwidth}{0.919}
\newenvironment{proofindent}{\begin{list}{}{\setlength{\leftmargin}{0.081\linewidth}}\item[]}{\end{list}}

\begin{document}

\title[A Relational Axiomatic Framework for the Foundations of Mathematics]
      {A Relational Axiomatic Framework \\
       for the Foundations of Mathematics}

\author{Lidia Obojska}

\date{18 December 2009}

\address{Lidia Obojska \\
         University of Podlasie, Dept. of Mathematics and Physics, ul. 3 Maja 54, 08-110 Siedlce, Poland}
\curraddr{Istituto Universitario Sophia, San Vito 28, 50064 Incisa Val D'Arno (FI), Italy}
\email{obojskal@ap.siedlce.pl}




\begin{abstract}
    We propose a Relational Calculus based on the concept of unary relation.
    In this Relational Calculus different axiomatic systems converge to a model called \emph{Dynamic Generative System} with 
    Symmetry (\emph{DGSS}).
    In \emph{DGSS} we define the concepts of relational set and function and prove that extensionality and the substitution 
    property of equality are theorems of \emph{DGSS}.

    As a first exemplification of \emph{DGSS}, we construct a model of natural numbers without relying on Peano's Axioms.
    Eventually, some new clarifications regarding the nature of the number zero are given.
\end{abstract}

\maketitle


\section{Introduction}%
\label{se:Introduction}

Kurt G\"{o}del \cite{Godel1940} thought that set theoretic antinomies are perhaps not problems connected to set theory itself but 
to the theory of concepts.
In this paper we would like to focus on the concept of \emph{unary relation}.
In mathematical literature the term \emph{unary relation} is applied to subsets of a given set \cite{Rasiowa2005}.

Thus we can think of mathematical objects as if they were either sets or relations (unary, binary, etc.).
For example, a function $f$ for a given argument $x$ can be seen as a set consisting of ordered pairs 
$\left(x, f\left(x\right)\right)$ or as a binary relation such that for every argument $x$ there exists exactly one term $y$ 
with $f\left(x\right) = y$.
Additionally, beginning with the work of Church, Turing and others \cite{Curry1958}, the concept of functions as algorithms was 
developed.

In 1996 E. De Giorgi introduced the concept of \emph{Fundamental Relation} \cite{DeGiorgi1996}, \cite{Forti1997} which describes 
the phenomenon of autoreference.
Generally, Fundamental Relations change the perspective of how we consider entities.
In a certain sense, a fundamental relation is a relation abstraction describing the entities related and the relation between them.
For this reason the present theory might be thought of as a kind of a \emph{Combinatory Logic} \cite{Curry1977}, but it also calls 
to mind \emph{Mereology} \cite{Lesniewski1927}, the theory of part-whole relationships.
In the context of Fundamental Relations, a unary relation makes sense if and only if there exists a fundamental binary relation 
which describes its behavior.

Fundamental Relations permit us to introduce a specific kind of unary relation, which exhibits dynamic behavior.
As a result a Relational Calculus can be defined by the intuitive concept of binary relation.

Starting with a simple system of two axioms, which describes a kind of dynamic identity, we then express this system in algebraic 
language and show that some equational rules and Peano's Axioms become theorems of this system.


\section{Fundamental Relations}%
\label{se:FundamentalRelations}

Adopting De Giorgi's way of considering relations, which he sees as fundamental entities \cite{DeGiorgi1996}, we restrict 
ourselves to restate only those definitions and axioms which are strictly necessary in the context of this article.
Moreover, we modify the notation in order to introduce some new concepts.

Let us begin with two primitive notions: quality and binary relations.

\begin{enumerate}
\setlength{\itemsep}{2mm}
    \item We will say that if $q$ is a quality (written as $\Q{q}$), then $\relu{q}{x}$ means that the object $x$ has the quality 
          $q$.
    \item Given two objects $x$, $y$ of any nature and a binary relation $r$, we will write $\relb{r}{x}{y}$ to say that 
          \textquotedblleft{}$x$ and $y$ are in relation $r$\textquotedblright{} or 
          \textquotedblleft{}$x$ is in relation $r$ with $y$\textquotedblright{}.
\end{enumerate}

We can proceed, as do De Giorgi et al. \cite{DeGiorgi1996}, and introduce the Fundamental Relation $\RQ$ which is defined as a 
\emph{binary relation}:

\begin{axiom}\label{ax1}
    \parbox[t]{0.75\linewidth}{
        $\RQ$ is a binary relation.
        \begin{enumerate}
            \setlength{\itemsep}{2mm}
            \item If $\relb{\RQ}{x}{y}$ then $\relu{\Q}{x}$.
            \item If $\relu{\Q}{q}$ then $\relb{\RQ}{q}{x} \equiv \relu{q}{x}$.
        \end{enumerate}}
\end{axiom}

Thus, if $\RQ$ is defined as a binary relation, it follows that $q$ can be considered a \emph{unary relation} \cite{Obojska2007}.

\begin{definition}\label{df1}
    A unary relation is any relation $\ast$ such that $\R\ast$ is a binary relation.
\end{definition}

To simplify the notation we will write: $\app{q}{x}$ instead of $\relb{\RQ}{q}{x}$ (i.~e. $\relb{\RQ}{q}{x} \equiv \app{q}{x}$).

In general, in the expression of the form $\app{x}{y}$, $\apl{~}$ will be used to indicate the relation abstraction or 
Fundamental Relation (in this case binary) connecting the unary relation $x$ and its argument $y$.
Because unary relations underlie those binary relations in which at least one of the objects is defined in terms of the other, 
unary relations we can think of as \textquotedblleft{}qualifying\textquotedblright{} their arguments.
In this perspective, qualities are viewed as relational entities, in so far as they can only be conceived of being 
\textquotedblleft{}in\textquotedblright{} some object (see Axiom~\ref{ax1}).


\section{Relational Calculus}%
\label{se:RelationalCalculus}

\subsection{Basic Definitions}%
\label{se:RelationalCalculus:BasicDefinitions}

As presented in Section~\ref{se:FundamentalRelations}, at the basis of our relational system is a primitive called a 
Fundamental Relation $\apl{~}$.
This Fundamental Relation is a kind of relation abstraction which can act as a combinator or as a pure relation.
When we write $\app{x}{y}$, we can read it as the application of $x$ onto $y$, as the process $x$ acting on the input $y$, or 
simply as the connection between $x$ and $y$.
$x$ might be seen as a quality in the sense of De Giorgi, but can also be any arbitrary entity in relation with $y$.
$\app{x}{y}$ \textquotedblleft{}creates\textquotedblright{} a new object, in the sense that $\app{x}{y}$ is a whole which turns 
out to consist of two related entities $x$ and $y$.
This interpretation gives a great freedom, as we will see below.
The notation $\app{x}{y}$ allows us to introduce any kind of unary relation in the place of $x$, even though this concept might 
not be intuitive.

We define a logic with unrestricted quantification, bearing in mind that quantifiers are abbreviations for certain quantifier-free 
expressions.
In the present system, well formed formulas are those formed from atomic constants, i.~e., the logical operators $\FORALL$, 
$\EXISTS$, $\AND$, $\OR$, $\NOT$, $\IMPLIES$, $\EQUALS$ and $=$, parentheses $[~]$, $\lbrace~\rbrace$ and variables, possibly 
connected by means of application $\apl{~}$.

A variable is any object which is able to enter into relation with other objects due to the application operator $\apl{~}$.

The nature of the objects remains explicitly open, as long as these objects are capable of being in relation with other objects 
in accordance with the Association Rule to be defined in Axiom~\ref{AR}.
The objects themselves can even be relations.

\begin{definition}\label{df2}
    The language of Relational Calculus (\emph{RC}) terms is built from an infinite number of variables: $x$, $y$, $z$, \dots 
    using the application operator $\apl{~}$ as follows:
    \begin{enumerate}
    \setlength{\itemsep}{2mm}
        \item If $x$ is a variable, then $x$ is an \emph{RC} term,
        \item if $x$, $y$ are variables, then $\app{x}{y}$ is an \emph{RC} term,
        \item if $x$ is a variable and $M$ is an \emph{RC} term, then $\app{M}{x}$ and $\app{x}{M}$ are \emph{RC} terms.
    \end{enumerate}
\end{definition}

\begin{axiom}\label{AR}
    $\FORALL a, b, c: \appp{a}{b}{c} = \app{\app{a}{b}}{c} = \app{a}{\app{b}{c}}$
\end{axiom}

Axiom~\ref{AR} describes a simple Association Rule used in Mathematics.
As stated above, the application operator allows us to consider $\appp{a}{b}{c}$ as a single object without constraining the view 
of its inner structure.
Hence, one can either find $\app{a}{b}$ to be in relation with $c$ or $a$ in relation with $\app{b}{c}$. 

Furthermore, we will use the classical \emph{Deduction Rules} for \textquotedblleft{}$=$\textquotedblright{}:

\begin{axiom}\label{DR}
    \setlength{\tabcolsep}{0mm}
    \begin{tabular}[t]{ll}
        $\FORALL p, q, r:\:$ & $\BLOCK{p = p}$, \\
                             & $\BLOCK{p = q} \IMPLIES \BLOCK{q = p}$, \\
                             & $\lbrace \BLOCK{p = q} \AND \BLOCK{q = r}\rbrace \IMPLIES \BLOCK{p = r}$
    \end{tabular}
\end{axiom}

\subsection{Dynamic Identity Triple (\emph{DIT})}%
\label{se:RelationalCalculus:DIT}

The Dynamic Identity Triple (\emph{DIT}) \cite{Obojska2007} is composed of three specific unary relations: 
\emph{DIT} $= \BLOCK{\xbar, \ybar, \zbar}$ with $\xbar \neq \ybar$, $\ybar \neq \zbar$, $\zbar \neq \xbar$, satisfying the
following two axioms:

\begin{axiom}\label{ax6}
    $\app{\xbar}{\ybar} = \ybar$
\end{axiom}

\begin{axiom}\label{ax7}
    $\app{\zbar}{\ybar} = \xbar$
\end{axiom}

Axiom~\ref{ax6} can be understood as a Distinction Rule (object $\ybar$ is separated from object $\xbar$) or as a kind of 
relation under which $\ybar$ remains invariant.

Axiom~\ref{ax7} describes the process of returning to $\xbar$: $\ybar$ returns to $\xbar$ via $\zbar$.

\begin{lemma}\label{Lzxz}
    $\app{\zbar}{\xbar} = \zbar$
\end{lemma}

\begin{proof}
    \parbox[t]{\proofwidth\linewidth}{
        Assume $\app{\zbar}{\xbar} \neq \zbar$: \\
        $\app{\zbar}{\ybar} \neq \app{\app{\zbar}{\xbar}}{\ybar} =\ax{AR} \app{\zbar}{\app{\xbar}{\ybar}} =\ax{ax6} 
        \app{\zbar}{\ybar}$ --- contradiction}
\end{proof}

\begin{proposition}\label{ER}
    An Extensional Rule $\FORALL p, q, x: \BLOCK{\app{p}{x} = \app{q}{x}} \IMPLIES \BLOCK{p = q}$  is a theorem of \emph{DIT}.
\end{proposition}

\begin{proof} We prove all possible combinations of $p$, $q$ and $x$:
\begin{proofindent}
    \begin{enumerate}
    \setlength{\itemsep}{2mm}
        \item $\BLOCK{\app{\xbar}{\xbar} = \app{\ybar}{\xbar}} \IMPLIES \BLOCK{\xbar = \ybar}$: \\
              $\xbar =\ax{ax7} \app{\zbar}{\ybar} =\lm{Lzxz} \app{\app{\zbar}{\xbar}}{\ybar} =\ax{AR} 
               \app{\zbar}{\app{\xbar}{\ybar}} =\ax{ax6} \app{\zbar}{\app{\xbar}{\app{\xbar}{\ybar}}} =\ax{AR}$ \\
              $\app{\zbar}{\app{\app{\xbar}{\xbar}}{\ybar}} = \app{\zbar}{\app{\app{\ybar}{\xbar}}{\ybar}} =\ax{AR} 
               \app{\zbar}{\app{\ybar}{\app{\xbar}{\ybar}}} =\ax{ax6} \app{\zbar}{\app{\ybar}{\ybar}} =\ax{AR}$ \\
              $\app{\app{\zbar}{\ybar}}{\ybar} =\ax{ax7} \app{\xbar}{\ybar} =\ax{ax6} \ybar$

        \item $\BLOCK{\app{\xbar}{\xbar} = \app{\zbar}{\xbar}} \IMPLIES \BLOCK{\xbar = \zbar}$: \\
              $\xbar =\ax{ax7} \app{\zbar}{\ybar} =\lm{Lzxz} \app{\app{\zbar}{\xbar}}{\ybar} = 
               \app{\app{\xbar}{\xbar}}{\ybar} =\ax{AR} \app{\xbar}{\app{\xbar}{\ybar}} =\ax{ax6} \app{\xbar}{\ybar} =\ax{ax6} 
               \ybar$ \\
              $\xbar =\ax{ax7} \app{\zbar}{\ybar} = \app{\zbar}{\xbar} =\lm{Lzxz} \zbar$

        \item $\BLOCK{\app{\ybar}{\xbar} = \app{\zbar}{\xbar}} \IMPLIES \BLOCK{\ybar = \zbar}$: \\
              $\xbar =\ax{ax7} \app{\zbar}{\ybar} =\lm{Lzxz} \app{\app{\zbar}{\xbar}}{\ybar} = 
               \app{\app{\ybar}{\xbar}}{\ybar} =\ax{AR} \app{\ybar}{\app{\xbar}{\ybar}} =\ax{ax6} \app{\ybar}{\ybar}$ \\
              $\ybar =\ax{ax6} \app{\xbar}{\ybar} = \app{\app{\ybar}{\ybar}}{\ybar} =\ax{AR} \app{\ybar}{\app{\ybar}{\ybar}} = 
               \app{\ybar}{\xbar} = \app{\zbar}{\xbar} =\lm{Lzxz} \zbar$

        \item $\BLOCK{\app{\xbar}{\ybar} = \app{\ybar}{\ybar}} \IMPLIES \BLOCK{\xbar = \ybar}$: \\
              $\xbar =\ax{ax7} \app{\zbar}{\ybar} =\ax{ax6} \app{\zbar}{\app{\xbar}{\ybar}} = 
               \app{\zbar}{\app{\ybar}{\ybar}} =\ax{AR} \app{\app{\zbar}{\ybar}}{\ybar} =\ax{ax7} \app{\xbar}{\ybar} =\ax{ax6} 
               \ybar$

        \item $\BLOCK{\app{\xbar}{\ybar} = \app{\zbar}{\ybar}} \IMPLIES \BLOCK{\xbar = \zbar}$: \\
              $\ybar =\ax{ax6} \app{\xbar}{\ybar} = \app{\zbar}{\ybar} =\ax{ax7} \xbar$ \\
              $\xbar =\ax{ax7} \app{\zbar}{\ybar} = \app{\zbar}{\xbar} =\lm{Lzxz} \zbar$

        \item $\BLOCK{\app{\ybar}{\ybar} = \app{\zbar}{\ybar}} \IMPLIES \BLOCK{\ybar = \zbar}$: \\
              $\ybar =\ax{ax6} \app{\xbar}{\ybar} =\ax{ax7} \app{\app{\zbar}{\ybar}}{\ybar} =\ax{AR} 
               \app{\zbar}{\app{\ybar}{\ybar}} = \app{\zbar}{\app{\zbar}{\ybar}} =\ax{ax7} \app{\zbar}{\xbar} =\lm{Lzxz} \zbar$

        \item $\BLOCK{\app{\xbar}{\zbar} = \app{\ybar}{\zbar}} \IMPLIES \BLOCK{\xbar = \ybar}$: \\
              $\app{\xbar}{\zbar} =\ax{ax7} \app{\app{\zbar}{\ybar}}{\zbar} =\ax{AR} \app{\zbar}{\app{\ybar}{\zbar}} = 
               \app{\zbar}{\app{\xbar}{\zbar}} =\ax{AR} \app{\app{\zbar}{\xbar}}{\zbar} =\lm{Lzxz} \app{\zbar}{\zbar}$ \\
              $\zbar =\lm{Lzxz} \app{\zbar}{\xbar} =\ax{ax7} \app{\zbar}{\app{\zbar}{\ybar}} =\ax{AR} 
               \app{\app{\zbar}{\zbar}}{\ybar} = \app{\app{\xbar}{\zbar}}{\ybar} =\ax{AR} 
               \app{\xbar}{\app{\zbar}{\ybar}} =\ax{ax7}$ \\
              $\app{\xbar}{\xbar}$ \\
              $\zbar =\lm{Lzxz} \app{\zbar}{\xbar} = \app{\app{\xbar}{\xbar}}{\xbar} =\ax{AR} \app{\xbar}{\app{\xbar}{\xbar}} = 
               \app{\xbar}{\zbar}$ \\
              $\xbar =\ax{ax7} \app{\zbar}{\ybar} = \app{\app{\xbar}{\zbar}}{\ybar} =\ax{AR} 
               \app{\xbar}{\app{\zbar}{\ybar}} =\ax{ax7} \app{\xbar}{\xbar} = \zbar$ \\
              $\xbar =\ax{ax7} \app{\zbar}{\ybar} = \app{\xbar}{\ybar} =\ax{ax6} \ybar$

        \item $\BLOCK{\app{\xbar}{\zbar} = \app{\zbar}{\zbar}} \IMPLIES \BLOCK{\xbar = \zbar}$: \\
              $\zbar =\lm{Lzxz} \app{\zbar}{\xbar} =\ax{ax7} \app{\zbar}{\app{\zbar}{\ybar}} =\ax{AR} 
               \app{\app{\zbar}{\zbar}}{\ybar} = \app{\app{\xbar}{\zbar}}{\ybar} =\ax{AR} 
               \app{\xbar}{\app{\zbar}{\ybar}} =\ax{ax7}$ \\
              $\app{\xbar}{\xbar}$ \\
              $\xbar =\ax{ax7} \app{\zbar}{\ybar} = \app{\app{\xbar}{\xbar}}{\ybar} =\ax{AR} 
               \app{\xbar}{\app{\xbar}{\ybar}} =\ax{ax6} \app{\xbar}{\ybar} =\ax{ax6} \ybar$ \\
              $\xbar =\ax{ax7} \app{\zbar}{\ybar} = \app{\zbar}{\xbar} =\lm{Lzxz} \zbar$

        \item $\BLOCK{\app{\ybar}{\zbar} = \app{\zbar}{\zbar}} \IMPLIES \BLOCK{\ybar = \zbar}$: \\
              $\xbar =\ax{ax7} \app{\zbar}{\ybar} =\ax{ax6} \app{\zbar}{\app{\xbar}{\ybar}} =\ax{AR} 
               \app{\app{\zbar}{\xbar}}{\ybar} =\ax{ax7} \app{\app{\zbar}{\app{\zbar}{\ybar}}}{\ybar} =\ax{AR}$ \\
              $\app{\app{\app{\zbar}{\zbar}}{\ybar}}{\ybar} = \app{\app{\app{\ybar}{\zbar}}{\ybar}}{\ybar} =\ax{AR} 
               \app{\app{\ybar}{\app{\zbar}{\ybar}}}{\ybar} =\ax{ax7} \app{\app{\ybar}{\xbar}}{\ybar} =\ax{AR}$ \\
              $\app{\ybar}{\app{\xbar}{\ybar}} =\ax{ax6} \app{\ybar}{\ybar}$ \\
              $\ybar =\ax{ax6} \app{\xbar}{\ybar} =\ax{ax7} \app{\app{\zbar}{\ybar}}{\ybar} =\ax{AR} 
                \app{\zbar}{\app{\ybar}{\ybar}} =\app{\zbar}{\xbar} =\lm{Lzxz} \zbar$
    \end{enumerate}
\end{proofindent}
\end{proof}

\begin{definition}\label{BR}
    \setlength{\tabcolsep}{2mm}
    \begin{tabular}[t]{ll}
        \multicolumn{2}{l}{We will call $r$ a bijective relation iff:} \\
        (1) & $\FORALL r, p, x: \BLOCK{\app{r}{x} = \app{p}{x}} \IMPLIES \BLOCK{r = p}$, \\
        (2) & $\FORALL r, x, y: \BLOCK{\app{r}{x} = \app{r}{y}} \IMPLIES \BLOCK{x = y}$.
    \end{tabular}
\end{definition}

Thus, by a \emph{bijective relation} we intend a  relation which obeys the extensional rule as well as the substitution 
property of equality.

\begin{proposition}\label{pr01}
    $\xbar$, $\ybar$, $\zbar$ are bijective relations.
\end{proposition}

\begin{proof}
\parbox[t]{\proofwidth\linewidth}{
    Aspect~(1) of Definition~\ref{BR} is assured by Proposition~\ref{ER}. \\
    We prove aspect~(2) for every possible combination of $r$, $p$ and $x$:}
\begin{proofindent}
    \begin{enumerate}
    \setlength{\itemsep}{2mm}
        \item $\BLOCK{\app{\xbar}{\xbar} = \app{\xbar}{\ybar}} \IMPLIES \BLOCK{\xbar = \ybar}$: \\
              $\xbar =\ax{ax7} \app{\zbar}{\ybar} =\ax{ax6} \app{\zbar}{\app{\xbar}{\ybar}} = 
               \app{\zbar}{\app{\xbar}{\xbar}} =\ax{AR} \app{\app{\zbar}{\xbar}}{\xbar} =\lm{Lzxz} \app{\zbar}{\xbar} =\lm{Lzxz} 
               \zbar$ \\
              $\xbar =\ax{ax7} \app{\zbar}{\ybar} = \app{\xbar}{\ybar} =\ax{ax6} \ybar$

        \item $\BLOCK{\app{\xbar}{\xbar} = \app{\xbar}{\zbar}} \IMPLIES \BLOCK{\xbar = \zbar}$: \\
              $\ybar =\ax{ax6} \app{\xbar}{\ybar} =\ax{ax6} \app{\xbar}{\app{\xbar}{\ybar}} =\ax{AR} 
               \app{\app{\xbar}{\xbar}}{\ybar} = \app{\app{\xbar}{\zbar}}{\ybar} =\ax{AR} 
               \app{\xbar}{\app{\zbar}{\ybar}} =\ax{ax7} \app{\xbar}{\xbar}$ \\
              $\xbar =\ax{ax7} \app{\zbar}{\ybar} = \app{\zbar}{\app{\xbar}{\xbar}} =\ax{AR} 
               \app{\app{\zbar}{\xbar}}{\xbar} =\lm{Lzxz} \app{\zbar}{\xbar} =\lm{Lzxz} \zbar$

        \item $\BLOCK{\app{\xbar}{\ybar} = \app{\xbar}{\zbar}} \IMPLIES \BLOCK{\ybar = \zbar}$: \\
              $\xbar =\ax{ax7} \app{\zbar}{\ybar} =\ax{ax6} \app{\zbar}{\app{\xbar}{\ybar}} = 
               \app{\zbar}{\app{\xbar}{\zbar}} =\ax{AR} \app{\app{\zbar}{\xbar}}{\zbar} =\lm{Lzxz} \app{\zbar}{\zbar}$ \\
              $\ybar =\ax{ax6} \app{\xbar}{\ybar} = \app{\app{\zbar}{\zbar}}{\ybar} =\ax{AR} 
               \app{\zbar}{\app{\zbar}{\ybar}} =\ax{ax7} \app{\zbar}{\xbar} =\lm{Lzxz} \zbar$

        \item $\BLOCK{\app{\ybar}{\xbar} = \app{\ybar}{\ybar}} \IMPLIES \BLOCK{\xbar = \ybar}$: \\
              $\ybar =\ax{ax6} \app{\xbar}{\ybar} =\ax{ax7} \app{\app{\zbar}{\ybar}}{\ybar} =\ax{AR} 
               \app{\zbar}{\app{\ybar}{\ybar}} = \app{\zbar}{\app{\ybar}{\xbar}} =\ax{AR} 
               \app{\app{\zbar}{\ybar}}{\xbar} =\ax{ax7}$ \\
              $\app{\xbar}{\xbar}$ \\
              $\xbar =\ax{ax7} \app{\zbar}{\ybar} = \app{\zbar}{\app{\xbar}{\xbar}} =\ax{AR} 
               \app{\app{\zbar}{\xbar}}{\xbar} =\lm{Lzxz} \app{\zbar}{\xbar} =\lm{Lzxz} \zbar$ \\
              $\xbar =\ax{ax7} \app{\zbar}{\ybar} = \app{\xbar}{\ybar} =\ax{ax6} \ybar$

        \item $\BLOCK{\app{\ybar}{\xbar} = \app{\ybar}{\zbar}} \IMPLIES \BLOCK{\xbar = \zbar}$: \\
              $\app{\ybar}{\ybar} =\ax{ax6} \app{\ybar}{\app{\xbar}{\ybar}} =\ax{AR} \app{\app{\ybar}{\xbar}}{\ybar} =
               \app{\app{\ybar}{\zbar}}{\ybar} =\ax{AR} \app{\ybar}{\app{\zbar}{\ybar}} =\ax{ax7} \app{\ybar}{\xbar}$ \\
              $\ybar =\ax{ax6} \app{\xbar}{\ybar} =\ax{ax7} \app{\app{\zbar}{\ybar}}{\ybar} =\ax{AR} 
               \app{\zbar}{\app{\ybar}{\ybar}} = \app{\zbar}{\app{\ybar}{\xbar}} =\ax{AR} 
               \app{\app{\zbar}{\ybar}}{\xbar} =\ax{ax7}$ \\
              $\app{\xbar}{\xbar}$ \\
              $\xbar =\ax{ax6} \app{\zbar}{\ybar} = \app{\zbar}{\app{\xbar}{\xbar}} =\ax{AR} 
               \app{\app{\zbar}{\xbar}}{\xbar} =\lm{Lzxz} \app{\zbar}{\xbar} =\lm{Lzxz} \zbar$

        \item $\BLOCK{\app{\ybar}{\ybar} = \app{\ybar}{\zbar}} \IMPLIES \BLOCK{\ybar = \zbar}$: \\
              $\ybar =\ax{ax6} \app{\xbar}{\ybar} =\ax{ax7} \app{\app{\zbar}{\ybar}}{\ybar} =\ax{AR} 
               \app{\zbar}{\app{\ybar}{\ybar}} = \app{\zbar}{\app{\ybar}{\zbar}} =\ax{AR} 
               \app{\app{\zbar}{\ybar}}{\zbar} =\ax{ax7} \app{\xbar}{\zbar}$ \\
              $\xbar =\ax{ax7} \app{\zbar}{\ybar} = \app{\zbar}{\app{\xbar}{\zbar}} =\ax{AR} 
               \app{\app{\zbar}{\xbar}}{\zbar} =\lm{Lzxz} \app{\zbar}{\zbar}$ \\
              $\ybar =\ax{ax6} \app{\xbar}{\ybar} = \app{\app{\zbar}{\zbar}}{\ybar} =\ax{AR} 
               \app{\zbar}{\app{\zbar}{\ybar}} =\ax{ax7} \app{\zbar}{\xbar} =\lm{Lzxz} \zbar$

        \item $\BLOCK{\app{\zbar}{\xbar} = \app{\zbar}{\ybar}} \IMPLIES \BLOCK{\xbar = \ybar}$: \\
              $\xbar =\ax{ax7} \app{\zbar}{\ybar} =\lm{Lzxz} \app{\app{\zbar}{\xbar}}{\ybar} = 
               \app{\app{\zbar}{\ybar}}{\ybar} =\ax{ax7} \app{\xbar}{\ybar} =\ax{ax6} \ybar$

        \item $\BLOCK{\app{\zbar}{\xbar} = \app{\zbar}{\zbar}} \IMPLIES \BLOCK{\xbar = \zbar}$: \\
              $\xbar =\ax{ax7} \app{\zbar}{\ybar} =\lm{Lzxz} \app{\app{\zbar}{\xbar}}{\ybar} = 
               \app{\app{\zbar}{\zbar}}{\ybar} =\ax{AR} \app{\zbar}{\app{\zbar}{\ybar}} =\ax{ax7} \app{\zbar}{\xbar} =\lm{Lzxz} 
               \zbar$

        \item $\BLOCK{\app{\zbar}{\ybar} = \app{\zbar}{\zbar}} \IMPLIES \BLOCK{\ybar = \zbar}$: \\
              $\ybar =\ax{ax6} \app{\xbar}{\ybar} =\ax{ax7} \app{\app{\zbar}{\ybar}}{\ybar} = 
               \app{\app{\zbar}{\zbar}}{\ybar} =\ax{AR} \app{\zbar}{\app{\zbar}{\ybar}} =\ax{ax7} \app{\zbar}{\xbar} =\lm{Lzxz} 
               \zbar$
    \end{enumerate}
\end{proofindent}
\end{proof}

Note that $\xbar$, $\ybar$, $\zbar$ must be distinct.
If they were not distinct, our model would either collapse to a simple formula similar to a classical definition of identity, 
$\app{\xbar}{\xbar} = \xbar$, or we would separate $\xbar$ from $\ybar$:

\begin{enumerate}
\setlength{\itemsep}{2mm}
    \item If $\xbar = \ybar$, then 
          $\BLOCK{\app{\xbar}{\xbar} =\ax{ax6} \xbar}$ and $\BLOCK{\app{\zbar}{\xbar} =\ax{ax7} \xbar}$; \\
          with $\BLOCK{\app{\xbar}{\xbar} = \app{\zbar}{\xbar}} \IMPLIES\pr{ER} \BLOCK{\xbar = \zbar}$ we get
          $\xbar = \ybar = \zbar$ and $\app{\xbar}{\xbar} = \xbar$.

    \item If $\xbar = \zbar$, then 
          $\BLOCK{\app{\xbar}{\ybar} =\ax{ax7} \xbar}$ and $\BLOCK{\app{\xbar}{\ybar} =\ax{ax6} \ybar}$, which implies
          $\BLOCK{\xbar = \ybar}$; \\
          again we get $\xbar = \ybar = \zbar$ and $\app{\xbar}{\xbar} = \xbar$.

    \item If $\ybar = \zbar$, then 
          $\BLOCK{\app{\xbar}{\ybar} =\ax{ax6} \ybar}$ and $\BLOCK{\app{\ybar}{\ybar} =\ax{ax7} \xbar}$, we get \\
          $\xbar =\ax{ax7} \app{\ybar}{\ybar} =\ax{ax6} \app{\app{\xbar}{\ybar}}{\ybar} =\ax{AR} 
           \app{\xbar}{\app{\ybar}{\ybar}} =\ax{ax7} \app{\xbar}{\xbar}$ and \\
          $\ybar =\ax{ax6} \app{\xbar}{\ybar} =\ax{ax7} \app{\app{\ybar}{\ybar}}{\ybar} =\ax{AR} \appp{\ybar}{\ybar}{\ybar}$, \\
          which leads to the eternal progression \\
          $\ybar = \appp{\ybar}{\ybar}{\ybar} =\ax{ax6} \appp{\appp{\ybar}{\ybar}{\ybar}}{\ybar}{\ybar} =\ax{AR} 
           \appppp{\ybar}{\ybar}{\ybar}{\ybar}{\ybar} = \appppppp{\ybar}{\ybar}{\ybar}{\ybar}{\ybar}{\ybar}{\ybar} \dots$ and \\
          $\xbar = \app{\ybar}{\ybar} =\ax{ax6} \app{\app{\xbar}{\ybar}}{\ybar} =\ax{AR} \appp{\xbar}{\ybar}{\ybar} =\ax{ax7} 
           \appp{\app{\ybar}{\ybar}}{\ybar}{\ybar} =\ax{AR} \apppp{\ybar}{\ybar}{\ybar}{\ybar} = \dots$
\end{enumerate}

Moreover, $\xbar$, $\ybar$, $\zbar$ are considered to be \textquotedblleft{}co-essential\textquotedblright{}, in the sense that
two relations alone, $\xbar$ and $\ybar$ or $\xbar$ and $\zbar$ or $\ybar$ and $\zbar$, cannot define \emph{DIT}.

Finally, the bijective property of $\xbar$, $\ybar$, $\zbar$ assures their dynamic character.
\emph{DIT} $\BLOCK{\xbar, \ybar, \zbar}$ defined in terms of $\xbar$, $\ybar$, $\zbar$ is fully dynamic.
For this reason we have called our model \emph{Dynamic Identity Triple}.

\subsection{Dynamic Identity Triple with Symmetry (\emph{DITS})}%
\label{se:RelationalCalculus:DITS}

Let us slightly modify \emph{DIT} and add a symmetry condition on $\zbar$ $\BLOCK{\app{\ybar}{\zbar} = \app{\zbar}{\ybar}}$.

\begin{lemma}\label{Lxzz}
    $\BLOCK{\app{\ybar}{\zbar} = \app{\zbar}{\ybar}} \IMPLIES \BLOCK{\app{\xbar}{\zbar} = \zbar}$
\end{lemma}

\begin{proof}
    $\app{\xbar}{\zbar} =\ax{ax7} \app{\app{\zbar}{\ybar}}{\zbar} =\ax{AR} \app{\zbar}{\app{\ybar}{\zbar}} = 
     \app{\zbar}{\app{\zbar}{\ybar}} =\ax{ax7} \app{\zbar}{\xbar} =\lm{Lzxz} \zbar$
\end{proof}

\begin{lemma}\label{Lxyyx}
    $\BLOCK{\app{\ybar}{\zbar} = \app{\zbar}{\ybar}} \IMPLIES \BLOCK{\app{\xbar}{\ybar} = \app{\ybar}{\xbar}}$
\end{lemma}

\begin{proof}
    $\app{\xbar}{\ybar} =\ax{ax7} \app{\app{\zbar}{\ybar}}{\ybar} = \app{\app{\ybar}{\zbar}}{\ybar} =\ax{AR} 
     \app{\ybar}{\app{\zbar}{\ybar}} =\ax{ax7} \app{\ybar}{\xbar}$
\end{proof}

\begin{lemma}\label{Lxxx}
    $\BLOCK{\app{\ybar}{\zbar} = \app{\zbar}{\ybar}} \IMPLIES \BLOCK{\app{\xbar}{\xbar} = \xbar}$
\end{lemma}

\begin{proof}
    $\xbar =\ax{ax7} \app{\zbar}{\ybar} =\lm{Lxzz} \app{\app{\xbar}{\zbar}}{\ybar} =\ax{AR} 
     \app{\xbar}{\app{\zbar}{\ybar}} =\ax{ax7} \app{\xbar}{\xbar}$
\end{proof}

Lemma~\ref{Lxxx} provides a useful extension to Axiom~\ref{ax6}.

Because the symmetry condition $\BLOCK{\app{\ybar}{\zbar} = \app{\zbar}{\ybar}}$ is not a theorem of \emph{DIT}, we define a 
Dynamic Identity Triple with Symmetry (\emph{DITS}) by adding another axiom:

\begin{axiom}\label{ax7a}
    $\app{\ybar}{\zbar} = \app{\zbar}{\ybar}$
\end{axiom}

\begin{definition}
    A Dynamic Identity Triple with Symmetry \emph{DITS} is a model satisfying Axioms~\ref{ax6},~\ref{ax7} and~\ref{ax7a}.
\end{definition}

\subsection{Dynamic Generative System (\emph{DGS})}%
\label{se:RelationalCalculus:DGS}

We can \textquotedblleft{}open\textquotedblright{} \emph{DIT} and assume that Axioms~\ref{ax6} and~\ref{ax7} are true for any 
variable $y$.
Using classical quantification rules, we obtain a modified model, which we will call \emph{Dynamic Generative System} -- 
\emph{DGS} $= \BLOCK{\xbar, y, z}$, with the following axioms:

\begin{axiom}\label{ax8}
    $\FORALL y: \app{\xbar}{y} = y$
\end{axiom}

\begin{axiom}\label{ax9}
    $\FORALL y \EXISTS z: \app{z}{y} = \xbar$
\end{axiom}

Analogously to Lemma~\ref{Lzxz}, we obtain:

\begin{lemma}\label{Lrxr}
    $\FORALL r: \app{r}{\xbar} = r$
\end{lemma}

\begin{proof}
\parbox[t]{\proofwidth\linewidth}{
    Assume $\FORALL r: \app{r}{\xbar} \neq r$: \\
    $\FORALL r, y: \app{r}{y} \neq \app{\app{r}{\xbar}}{y} =\ax{AR} \app{r}{\app{\xbar}{y}} =\ax{ax8} \app{r}{y}$ --- 
    contradiction}
\end{proof}

Note that Lemma~\ref{Lrxr} does not exclude the possibility of $r = \xbar$.

Hence, we get $\app{\xbar}{\xbar} = \xbar$ by Axioms~\ref{ax8} and~\ref{ax9} or by Axioms~\ref{ax6},~\ref{ax7} and the 
symmetry condition (Axiom~\ref{ax7a}).

\subsection{Dynamic Generative System with Symmetry (\emph{DGSS})}%
\label{se:RelationalCalculus:DGSS}

As in the case of \emph{DIT}, we can add a similar symmetry condition to \emph{DGS}.

\begin{axiom}\label{ax9a}
    $\FORALL y \EXISTS z: \app{z}{y} = \app{y}{z} = \xbar$
\end{axiom}

\begin{definition}
    A Dynamic Generative System with Symmetry (\emph{DGSS}) is a model that satisfies Axioms~\ref{ax8},~\ref{ax9} and~\ref{ax9a}.
\end{definition}

\begin{lemma}\label{lm2a}
    $\FORALL x, y: \BLOCK{x = y} \IMPLIES \EXISTS s, t: \BLOCK{\app{s}{x} = \app{t}{y} = \xbar} \AND \BLOCK{s = t}$
\end{lemma}

\begin{proof}
\parbox[t]{\proofwidth\linewidth}{
    The existence of $s$ and $t$ is assured by Axiom~\ref{ax9}: \\
    $\IMPLIES\ax{ax9} \EXISTS s: \xbar = \app{s}{x}$ \\
    $\IMPLIES\ax{ax9} \EXISTS t: \xbar = \app{t}{y}$ \\[2mm]
    Furthermore, Axiom~\ref{ax9} guarantees the equivalence of $\app{s}{x}$ and $\app{t}{y}$: \\
    $\app{s}{y} = \app{s}{x} =\ax{ax9} \xbar =\ax{ax9} \app{t}{y} = \app{t}{x}$ \\[2mm]
    Finally, we prove the equivalence of $s$ and $t$: \\
    $s =\lm{Lrxr} \app{s}{\xbar} = \app{s}{\app{t}{y}} =\ax{ax9a} \app{s}{\app{y}{t}} =\ax{AR} \app{\app{s}{y}}{t} = 
     \app{\xbar}{t} =\ax{ax8} t$}
\end{proof}

\begin{lemma}\label{lm2b}
    $\FORALL x, y: \BLOCK{x = y} \IMPLIES \EXISTS s, t: \BLOCK{\app{x}{s} = \app{y}{t} = \xbar} \AND \BLOCK{s = t}$
\end{lemma}

\begin{proof}
\parbox[t]{\proofwidth\linewidth}{
    The existence of $s$ and $t$ is assured by Axiom~\ref{ax9}: \\
    $\IMPLIES \ax{ax9} \EXISTS s: \xbar = \app{s}{x}$ \\
    $\IMPLIES \ax{ax9} \EXISTS t: \xbar = \app{t}{y}$ \\[2mm]
    Furthermore, Axiom~$\ref{ax9}$ guarantees the equivalence of $\app{s}{x}$ and $\app{t}{y}$: \\
    $\app{x}{s} =\ax{ax9a} \app{s}{x} =\ax{ax9} \xbar =\ax{ax9} \app{t}{y} =\ax{ax9a} \app{y}{t}$ \\[2mm]
    $s =\lm{Lrxr} \app{s}{\xbar} = \app{s}{\app{y}{t}} =\ax{AR} \app{\app{s}{y}}{t} = \app{\xbar}{t} = \ax{ax8} t$}
\end{proof}

\begin{corollary}
    The symmetry condition (Axiom~\ref{ax9a}) assures the uniqueness of $z$ in Axiom~\ref{ax9}: 
    $\FORALL y \EXISTSONE z: \BLOCK{\app{z}{y} = \xbar}$.
\end{corollary}

Moreover, we can prove the following properties of \emph{DGSS}:

\begin{lemma}\label{lm2c}
    $\FORALL x, y, z: \BLOCK{\app{z}{x} = \app{z}{y} = \xbar} \IMPLIES \BLOCK{x = y}$
\end{lemma}

\begin{proof}
    $x =\lm{Lrxr} \app{x}{\xbar} = \app{x}{\app{z}{y}} =\ax{AR} \app{\app{x}{z}}{y} =\ax{ax9a} \app{\app{z}{x}}{y} = 
     \app{\xbar}{y} =\ax{ax8} y$
\end{proof}

Analogously:

\begin{lemma}\label{lm2d}
    $\FORALL x, y, z: \BLOCK{\app{x}{z} = \app{y}{z} = \xbar} \IMPLIES \BLOCK{x = y}$
\end{lemma}

\begin{proof}
    $\BLOCK{\app{x}{z} = \app{y}{z} = \xbar} \EQUALS\ax{ax9a} \BLOCK{\app{z}{x} = \app{z}{y} = \xbar} \IMPLIES\lm{lm2c} 
     \BLOCK{x = y}$
\end{proof}

Finally, we prove that in \emph{DGSS} extensionality and the substitution property of equality hold:

\begin{proposition}\label{pr2e}
    $\FORALL x, y, z: \BLOCK{\app{z}{x} = \app{z}{y}} \EQUALS \BLOCK{x = y}$
\end{proposition}

\begin{proof}
\parbox[t]{\proofwidth\linewidth}{
    \begin{enumerate}
    \setlength{\itemsep}{2mm}
        \item $\FORALL x, y, z: \BLOCK{\app{z}{x} = \app{z}{y}} \IMPLIES \BLOCK{x = y}$: \\
              $\BLOCK{\app{z}{x} = \app{z}{y}} \IMPLIES\lm{lm2a} 
               \EXISTS s: \BLOCK{\app{s}{\app{z}{x}} = \app{s}{\app{z}{y}} = \xbar}$ \\
              $\BLOCK{\app{z}{x} = \app{z}{y}} \IMPLIES 
               \BLOCK{\app{s}{\app{z}{x}} = \app{s}{\app{z}{y}} = \xbar} \EQUALS\ax{AR}$ \\
              $\BLOCK{\app{\app{s}{z}}{x} = \app{\app{s}{z}}{y} = \xbar} \IMPLIES\lm{lm2c} \BLOCK{x = y}$
        \item $\FORALL x, y, z: \BLOCK{x = y} \IMPLIES \BLOCK{\app{z}{x} = \app{z}{y}}$: \\
              $\IMPLIES\ax{ax9} \FORALL a \EXISTS s: \app{s}{a} = \xbar$ \\
              With $a = \app{z}{x}$: 
              $\xbar = \app{s}{a} = \app{s}{\app{z}{x}} =\ax{AR} \app{\app{s}{z}}{x} = \app{\app{s}{z}}{y}$ \\
              $\BLOCK{x = y} \IMPLIES\lm{lm2a} \BLOCK{\app{\app{s}{z}}{x} = \app{\app{s}{z}}{y} = \xbar} \EQUALS\ax{AR}$ \\
              $\BLOCK{\app{s}{\app{z}{x}} = \app{s}{\app{z}{y}} = \xbar} \IMPLIES\lm{lm2c} \BLOCK{\app{z}{x} = \app{z}{y}}$
    \end{enumerate}}
\end{proof}

Analogously:

\begin{proposition}\label{pr2f}
    $\FORALL x, y, z: \BLOCK{\app{x}{z} = \app{y}{z}} \EQUALS \BLOCK{x = y}$
\end{proposition}

\begin{proof}
\parbox[t]{\proofwidth\linewidth}{
    \begin{enumerate}
    \setlength{\itemsep}{2mm}
        \item $\FORALL x, y, z:\BLOCK{\app{x}{z} = \app{y}{z}} \IMPLIES \BLOCK{x = y}$: \\
              $\BLOCK{\app{x}{z} = \app{y}{z}} \IMPLIES\lm{lm2b} 
               \EXISTS s: \BLOCK{\app{\app{x}{z}}{s} = \app{\app{y}{z}}{s} = \xbar}$ \\
              $\BLOCK{\app{x}{z} = \app{y}{z}} \IMPLIES 
               \BLOCK{\app{\app{x}{z}}{s} = \app{\app{y}{z}}{s} = \xbar} \EQUALS\ax{AR}$ \\
              $\BLOCK{\app{x}{\app{z}{s}} = \app{y}{\app{z}{s}} = \xbar} \IMPLIES\lm{lm2d} \BLOCK{x = y}$

        \item $\FORALL x, y, z: \BLOCK{x = y} \IMPLIES \BLOCK{\app{x}{z} = \app{y}{z}}$: \\
              $\IMPLIES\ax{ax9} \FORALL a \EXISTS s: \app{s}{a} = \xbar =\ax{ax9a} \app{a}{s}$ \\
              With $a = \app{x}{z}$: 
              $\xbar = \app{a}{s} = \app{\app{x}{z}}{s} =\ax{AR} \app{x}{\app{z}{s}} = \app{y}{\app{z}{s}}$ \\
              $\BLOCK{x = y} \IMPLIES\lm{lm2b} \BLOCK{\app{x}{\app{z}{s}} = \app{y}{\app{z}{s}} = \xbar} \EQUALS\ax{AR}$ \\
              $\BLOCK{\app{\app{x}{z}}{s} = \app{\app{y}{z}}{s} = \xbar} \IMPLIES\lm{lm2d} \BLOCK{\app{x}{z} = \app{y}{z}}$
    \end{enumerate}}
\end{proof}


\section{Application in Foundations of Mathematics}%
\label{se:Application}

\subsection{Relational sets}
\label{se:Application:RelSets}

\begin{definition}\label{df3}
    $\set{Q}$ is a relational set of objects $x$ and $x$ is an element of $\set{Q}$ if for a specific quality $\dit{q}$ holds: 
    $\app{\dit{q}}{x} = x$ [$x \IN Q$, $Q =_{df} \app{\dit{q}}{x}$].
\end{definition}

\begin{example}\label{ex1}
    Let $\dit{n}$ be the quality of being natural number.
    Thus, $\app{\dit{n}}{x}$ defines the relational set of natural numbers $\set{N} = \app{\dit{n}}{x}$.
    If $x$ is a natural number, $\app{\dit{n}}{x} = x$ holds.
\end{example}

If $\xbar = y = \dit{q}$ then $\set{Q} =_{df} \app{\dit{q}}{\dit{q}} =\axs{ax8}{ax9} \dit{q} \IMPLIES \dit{q} \IN \dit{q}$, which 
can be interpreted as a singleton, a set composed of only one element: 
$\dit{q}$ is a relational set and it is an element of itself.

Having in mind Russell's paradox \cite{Russell1938}, one can ask what happens when $\dit{p}$ is considered to be the quality of 
not being an element of itself $\NOT \BLOCK{x \IN x}$.

Let $\set{P} =_{df} \app{\dit{p}}{x}$, $\BLOCK{x \IN \set{P}} \EQUALS \BLOCK{x = \app{\dit{p}}{x}}$.

$\NOT \BLOCK{x \IN x} \EQUALS \NOT \BLOCK{x \IN \set{P}} \EQUALS \BLOCK{x \neq \app{\dit{p}}{x}}$.

As a result there exist objects which are not relational sets.

At the beginning we said that our theory calls to mind mereology, a sort of collective set theory, first formulated by 
S. Le\'{s}niewski \cite{Lesniewski1927}.
Mereology is collective in the sense that a mereological \textquotedblleft{}set\textquotedblright{} is a whole (a collective
aggregate or class) composed of \textquotedblleft{}parts\textquotedblright{} and the fundamental relation is that of being a 
\textquotedblleft{}part\textquotedblright{} of the whole, an element of a class.
Thus being an element of a class is equivalent to being a subset (proper or improper) of a class.
In this sense it is clear that every class is an element of itself.

Let us now define the concept of subset:

\begin{definition}\label{df5}
    A relational set $\set{B}$ [$\set{B} = \app{\dit{b}}{x}$] is a relational subset of a relational set 
    $\set{A}$ $\BLOCK{\set{A} = \app{\dit{a}}{x}}$, written as $\BLOCK{\app{\dit{b}}{x} \SUBSET \app{\dit{a}}{x}}$, iff 
    $\FORALL x: \lbrace \BLOCK{\app{\dit{b}}{x} = x} \IMPLIES \BLOCK{\app{\dit{a}}{x} = x} \rbrace$.
\end{definition}

Let us verify the cases $\BLOCK{x = \dit{a}}$, $\BLOCK{x = \dit{b}}$ and $\BLOCK{x = \dit{a} = \dit{b}}$:

\begin{enumerate}
\setlength{\itemsep}{2mm}
    \item If $x = \dit{a}$ then $\app{\dit{b}}{\dit{a}} = \dit{a} =\lm{Lrxr} \app{\dit{a}}{\dit{a}} \IMPLIES 
                                 \lbrace \BLOCK{\dit{a} \IN \set{B}} \IMPLIES \BLOCK{\dit{a} \IN \set{A}} \rbrace$.
    \item If $x = \dit{b}$ then $\app{\dit{b}}{\dit{b}} = \dit{b} =\ax{ax8} \app{\dit{a}}{\dit{b}} \IMPLIES 
                                 \lbrace \BLOCK{\dit{b} \IN \set{B}} \IMPLIES \BLOCK{\dit{b} \IN \set{A}} \rbrace$.
    \item If $x = \dit{a} = \dit{b}$ then 
          $\BLOCK{\app{\dit{a}}{\dit{a}} = \dit{a} \IMPLIES \app{\dit{a}}{\dit{a}} = \dit{a}} \IMPLIES 
           \BLOCK{\dit{a} \IN \set{A}}$, and analogously $\BLOCK{\dit{b} \IN \set{B}}$.
\end{enumerate}

Now we can introduce the concept of function.

\begin{definition}\label{df5b}
    A binary relation $f: \set{A} \longmapsto \set{B}$ is a function iff \\
    $\FORALL x, y, z: \lbrace \BLOCK{x \IN \set{A}} \AND \BLOCK{y, z \IN \set{B}} \rbrace \IMPLIES 
     \lbrace \BLOCK{\appp{f}{x}{y} = \appp{f}{x}{z}} \IMPLIES \BLOCK{y = z} \rbrace$, \\
    where $\BLOCK{x \IN \set{A}} \equiv \BLOCK{\app{\dit{a}}{x} = x}$ and 
    $\BLOCK{y, z \IN \set{B}} \equiv \lbrace \BLOCK{\app{\dit{b}}{y} = y} \AND \BLOCK{\app{\dit{b}}{z} = z} \rbrace$.
\end{definition}

\subsection{Peano's Axioms as Theorems of \emph{DGSS}}
\label{se:Application:Peano}

The next step would be to see whether it is possible to redefine Peano's Axioms \cite{Peano1958} in terms of relational sets.

Let us take $\dit{n}$ instead of $\xbar$ in Axiom~\ref{ax8}, which stands for the quality of being natural number.

$\app{\dit{n}}{x}$ defines a relational set of natural numbers $\BLOCK{\set{N} = \app{\dit{n}}{x}, \FORALL x \neq \dit{n}}$.
We explicitly exclude $\dit{n} = x$, because \textquotedblleft{}$\,\dit{n}\,$\textquotedblright{}, the quality of being natural 
number should not be a number itself.
Furthermore, we define a smallest natural number and call it \textquotedblleft{}$1$\textquotedblright{}, in accordance with the 
original formulation of Peano \cite{Peano1958}.

Since $1$ is a natural number, $\app{\dit{n}}{1} =\ax{ax8} 1$ holds.

Then we define a successor function (in Mathematics considered as a unary operation) $\app{1}{x}$:
$\app{1}{x}$ is called a successor of $x$.

\begin{lemma}\label{lm3}
    $\app{1}{x}$ is a function of type $f: \set{N} \longmapsto \set{N}$.
\end{lemma}

\begin{proof}
\parbox[t]{\proofwidth\linewidth}{
    We have to prove that $\app{1}{x}$ fulfills the requirements of Definition~\ref{df5b}, i.~e. \\
    $\FORALL a, b, c: \BLOCK{\appp{f}{a}{b} = \appp{f}{a}{c}} \IMPLIES \BLOCK{b = c}$ with $f = 1$, $a = x$, $b = \app{1}{x}$, 
    $c = y$: \\[2mm]
    At first, we show $\FORALL x: \BLOCK{\app{\dit{n}}{x} = x} \IMPLIES \BLOCK{\app{\dit{n}}{\app{1}{x}} = \app{1}{x}}$: \\
    $\app{\dit{n}}{\app{1}{x}} =\ax{AR} \app{\app{\dit{n}}{1}}{x} =\ax{ax8} \app{1}{x}$ \\[2mm]
    We show now that $\FORALL x, y: \BLOCK{\appp{1}{x}{\app{1}{x}} = \app{1}{x}{y}} \IMPLIES \BLOCK{\app{1}{x} = y}$: \\
    $\FORALL x, y: \BLOCK{\appp{1}{x}{\app{1}{x}} = \appp{1}{x}{y}} \EQUALS \ax{AR} \BLOCK{\app{\app{1}{x}}{\app{1}{x}} = 
     \app{\app{1}{x}}{y}} \EQUALS\pr{pr2e}$ \\
    $\BLOCK{\app{1}{x} = y}$.}
\end{proof}

Note that terms like $\apppp{1}{1}{1}{1}$ are irreducible and can naturally be interpreted as numbers.

We claim:

\begin{proposition}\label{pr1}
    $\BLOCK{\app{\dit{n}}{x}, \app{1}{x}, 1}$ is a model of natural numbers.
\end{proposition}

\begin{proof}
    We prove that all Peano Axioms are theorems of \emph{DGSS}:

\begin{proofindent}
    \begin{enumerate}
    \setlength{\itemsep}{2mm}
        \item $\app{\dit{n}}{1} =\ax{ax8}1 \IMPLIES\df{df3} 1 \IN \set{N}$
        \item $\FORALL x: \BLOCK{\app{\dit{n}}{x} = x} \IMPLIES \BLOCK{\app{\dit{n}}{\app{1}{x}} = \app{1}{x}}$: \\
              $\app{\dit{n}}{\app{1}{x}} =\ax{AR} \app{\app{\dit{n}}{1}}{x} =_{(1)} \app{1}{x}$

        \item $\lbrace \FORALL x, y: \BLOCK{\app{\dit{n}}{x}=x} \AND \BLOCK{\app{\dit{n}}{y} = y} \AND 
                                     \BLOCK{\app{1}{x} = \app{1}{y}} \rbrace \IMPLIES \BLOCK{x = y}$: \\
              $\FORALL x, y: \BLOCK{\app{1}{x} = \app{1}{y}} \IMPLIES \pr{pr2e} \BLOCK{x = y}$
        \item $\FORALL x: \BLOCK{\app{\dit{n}}{x} = x} \IMPLIES \app{1}{x} \neq 1$: \\
              Assume $\EXISTS x: \app{1}{x} = 1$ for $x \neq \dit{n}$: \\
              $\app{\dit{n}}{1} =_{(1)} 1 =\lm{Lrxr} \app{1}{\dit{n}}$ \\
              $\BLOCK{\app{1}{x} = \app{1}{\dit{n}}} \IMPLIES\pr{pr2e} \BLOCK{x = \dit{n}}$ --- contradiction.
        \item $\lbrace \BLOCK{\set{M} \SUBSET \set{N}} \AND \BLOCK{1 \IN \set{M}} \AND 
               \lbrace \FORALL k: \BLOCK{k \IN \set{M}} \IMPLIES \BLOCK{\app{1}{k} \IN \set{M}} \rbrace \rbrace \IMPLIES 
               \BLOCK{\set{M} = \set{N}}$: \\
              With $\BLOCK{x \IN \set{M}} \equiv \BLOCK{\app{\dit{m}}{x} = x}$: \\
              $\app{\dit{m}}{1} = 1 =\lm{Lrxr} \app{1}{\dit{m}}$ \\
              $\FORALL k: \BLOCK{\app{\dit{m}}{k} = k} \EQUALS\pr{pr2e} 
                          \BLOCK{\app{1}{\app{\dit{m}}{k}} = \app{1}{k}} \EQUALS\ax{AR} 
                          \BLOCK{\app{\app{1}{\dit{m}}}{k} = \app{1}{k}} \EQUALS$ \\
                         $\BLOCK{\app{\app{\dit{m}}{1}}{k} = \app{1}{k}} \EQUALS \ax{AR} 
                          \BLOCK{\app{\dit{m}}{\app{1}{k}} = \app{1}{k}}$ \\
              $\IMPLIES_{(1)} \app{\dit{n}}{1} = 1$ \\
              $\IMPLIES_{(2)} \FORALL k: \BLOCK{\app{\dit{n}}{k} = k} \IMPLIES \BLOCK{\app{\dit{n}}{\app{1}{k}} = \app{1}{k}}$ \\
              $\FORALL k: \lbrace \BLOCK{\app{\dit{m}}{k} = k} \AND \BLOCK{\app{\dit{n}}{k} = k} \rbrace \IMPLIES\ax{DR} 
               \FORALL k: \BLOCK{\app{\dit{m}}{k} = \app{\dit{n}}{k}} \equiv\df{df3} \BLOCK{\set{M} = \set{N}}$
    \end{enumerate}
\end{proofindent}
\end{proof}

Natural numbers can be defined as follows:

\begin{enumerate}
\setlength{\itemsep}{2mm}
    \item $1$ is a natural number.
    \item The successor of $1$ is $\app{1}{1} =_{df} 2$
    \item $\app{1}{2} = \app{1}{\app{1}{1}} = \app{\app{1}{1}}{1} = \app{2}{1} =_{df} 3$
    \item $\app{1}{3} = \app{1}{\app{1}{2}} = \app{\app{1}{1}}{2} = \app{2}{2} = \app{2}{\app{1}{1}} = \app{\app{2}{1}}{1} = 
           \app{3}{1} =_{df} 4$ etc.
\end{enumerate}

\subsection{The Number Zero}%
\label{se:Application:Zero}

To reconstruct natural numbers we have been considering $\xbar$ in Axiom~\ref{ax8} as a kind of quality and we assumed that 
$\dit{n} = \xbar \neq x \:\FORALL x$.
Thus $\dit{n}$ should be different from all natural numbers.

We now define another fundamental number \textquotedblleft{}$0$\textquotedblright{} and put $\dit{n} = 0$.
We can verify that Proposition~\ref{pr1} holds for every $x \neq 0$:

\begin{enumerate}
\setlength{\itemsep}{2mm}
    \item $\app{0}{1} = 1$
    \item $\FORALL x \neq 0: \BLOCK{\app{0}{x} = x} \IMPLIES \BLOCK{\app{0}{\app{1}{x}} = \app{1}{x}}$
    \item $\lbrace \FORALL x, y \neq 0: \BLOCK{\app{0}{x} = x} \AND \BLOCK{\app{0}{y} = y} \AND 
                                        \BLOCK{\app{1}{x} = \app{1}{y}} \rbrace \IMPLIES \BLOCK{x = y}$
    \item $\FORALL x \neq 0: \BLOCK{\app{0}{x} = x} \IMPLIES \app{1}{x} \neq 1$
    \item $\lbrace \BLOCK{\app{\dit{m}}{k} \SUBSET \app{0}{k}} \AND \BLOCK{\app{\dit{m}}{1} = 1} \AND$ \\
          $\lbrace \FORALL k: \BLOCK{\app{\dit{m}}{k} = k} \IMPLIES 
                     \BLOCK{\app{\dit{m}}{\app{1}{k}} = \app{1}{k}} \rbrace \rbrace \IMPLIES \BLOCK{\dit{m} = 0}$
\end{enumerate}

If we would allow $x = 0$ in Proposition~\ref{pr1}, we would produce a contradiction in the fourth Peano Axiom: 
$\BLOCK{\app{1}{0} \neq 1}$ and $\BLOCK{\app{1}{0} = 1}$ by Lemma~\ref{Lrxr} would imply $\BLOCK{1 \neq 1}$.

Thanks to Axiom~\ref{ax8}, which states $\FORALL x: \app{0}{x} = x$ we can easily define all natural numbers beginning from $0$:

\begin{example}\label{ex4}
    \setlength{\tabcolsep}{0mm}
    \begin{tabular}[t]{l}
        $2$ is a natural number, hence $\app{0}{2} = 2$. \\
        $\app{0}{2} =_{df} \app{0}{\app{1}{1}} =\ax{AR} \app{\app{0}{1}}{1} =_{(1)} \app{1}{1} =_{df} 2$.
    \end{tabular}
\end{example}

Many mathematicians introduce $0$ as a first natural number, but in \emph{RC} $0$ cannot be considered a natural number.
Thanks to Lemma~\ref{Lrxr}, i.~e. $\app{1}{0} = 1$, read \textquotedblleft{}$1$ is the successor of $0$\textquotedblright{}, the 
intended meaning of $0$ and $1$ is fully preserved, and no additional axioms are needed.


\section{Conclusions}%
\label{se:Conclusions}

We found the concept of Fundamental Relation, as defined by E. De Giorgi, necessary but not sufficient to express a dynamic 
behavior of relations.
To do this, we have introduced three distinct unary relations which form a Dynamic Identity Triple (\emph{DIT}).
The three basic relations obey two very simple axioms and if desired a third axiom of symmetry (\emph{DITS}) can be added.
The relations in a model interact in a way that any two of them cannot \textquotedblleft{}operate\textquotedblright{} at the 
same time: a phenomenon which can be characterized by the word \emph{dynamic}.

Using classical rules of quantification, we modify the basic models and obtain the Dynamic Generative System (\emph{DGS}), 
which can also be enhanced by a symmetry condition (\emph{DGSS}).
After defining rules for a Relational Calculus, we define some basic set theoretic notions, the concept of function and find 
that Peano's Axioms and extensionality and the substitution property of equality are theorems of \emph{DGSS}.
In addition, it becomes clear that the number zero can be added to the system of Peano's Axioms, but cannot be considered a 
natural number.


\nocite{Engeler1992}
\nocite{Frege1884}

\end{document}